\documentclass[12pt]{article}
\usepackage[latin1]{inputenc}
\usepackage[english]{babel}
\usepackage{babel,indentfirst}
\usepackage{amsfonts,amssymb,amsmath}
\usepackage{latexsym}
\usepackage{graphicx}
\usepackage{lineno}
\usepackage{enumerate}

\title{On the domination of Cartesian product of directed cycles: Results for certain equivalence classes of lengths}

\author{Michel Mollard\thanks{CNRS Universit\'e Joseph Fourier}\\
\small Institut Fourier \\[-0.8ex]
\small 100, rue des Maths\\[-0.8ex]
\small 38402 St Martin d'Hères Cedex FRANCE \\
\small \texttt{michel.mollard@ujf-grenoble.fr}
}

\date{}

\newtheorem{defin}{Definition}
\newtheorem{lemm}[defin]{Lemma}
\newtheorem{theor}[defin]{Theorem}
\newtheorem{prop}[defin]{Proposition}

\newtheorem{conj}[defin]{Conjecture}

\newcommand{\qed}{\hspace{\stretch{1}}$\Box$}
\newenvironment{proof}{{\bf Proof :}}{\qed}

\begin{document}
\modulolinenumbers[2]
\linenumbers

 \maketitle

  \abstract{Let $\gamma(\overrightarrow{C_m}\Box \overrightarrow{C_n})$ be the domination number of the Cartesian product of directed cycles $\overrightarrow{C_m}$ and $\overrightarrow{C_n}$ for $m,n\geq2$. Shaheen \cite{Shaheen} and  Liu et al.(\cite{Liu1}, \cite{Liu2}) determined the value of $\gamma(\overrightarrow{C_m}\Box \overrightarrow{C_n})$ when $m \leq 6$ and \cite{Liu2} when both $m$ and $n$ $\equiv 0$ $(mod\: 3)$.
  In this article we give, in general, the value of $\gamma(\overrightarrow{C_m}\Box \overrightarrow{C_n})$ when $m\equiv 2$ $(mod\: 3)$ and improve the known lower bounds for most of the remaining cases. We also disprove the conjectured formula for the case $m$ $\equiv 0$ $(mod\: 3)$ appearing in \cite{Liu2}.
  \\
  
\textbf{AMS Classification[2010]}:05C69,05C38.\\

\textbf{Keywords}: Directed graph, Cartesian product, Domination number, Directed cycle.\\}

  \section{Introduction and definitions}
  Let $D=(V,A)$ be a finite directed graph (digraph for short) without loops or multiple arcs.
  
  A vertex $u$ \emph{dominates} a vertex $v$ if $u=v$ or $uv \in A$.
  A set $W\subseteq V$ is a \emph{dominating} set of $D$ if any vertex of $V$ is dominated by at least one vertex of $W$.
  The \emph{domination number} of $D$ , denoted by $\gamma(D)$ is the minimum cardinality of a dominating set. The set $V$ is a dominating set thus $\gamma(D)$ is finite. These definitions  extend to digraphs the classical domination notion for undirected graphs.
      
   The determination of  the domination number of  a directed or undirected graph is, in general, a difficult question in graph theory. Furthermore this problem has connections with information theory. For example the domination number of hypercubes is linked to error-correcting codes. 
   Among the lot of related works, Haynes et al. (\cite{ Haynes1}, \cite{ Haynes2}), mention the special cases of the domination of Cartesian products of undirected paths, cycles or more generally graphs(\cite{Chang} to \cite{Harntel}, \cite{Jacobson}, \cite{Kla}).
    
For two digraphs $D_1=(V_1,A_1)$ and $D_2=(V_2,A_2)$ the \emph{Cartesian product} $D_1\Box D_2$ is the digraph with vertex set $V_1\times V_2$ and $(x_1,x_2)(y_1,y_2) \in A(D_1\Box D_2)$ if and only if $x_1y_1 \in A_1$ and $x_2=y_2$ or $x_2y_2 \in A_2$ and $x_1=y_1$.
Note that $D_2\Box D_1$ is isomorphic to $D_1\Box D_2$.
In \cite{Shaheen} Shaheen determined the domination number of $\overrightarrow{C_m}\Box \overrightarrow{C_n}$ for $m \leq 6$ and arbitrary $n$. 
In two articles \cite{Liu1}, \cite{Liu2} Liu et al. considered independently the domination number of the Cartesian product of two directed cycles.  They gave also the value of $\gamma(\overrightarrow{C_m}\Box \overrightarrow{C_n})$ when $m \leq 6$ and when both $m$ and $n$ $\equiv 0$ $(mod\: 3)$ \cite{Liu2}. Furthermore they proposed lower and upper bounds for the general case.

In this paper we are able to give, in general, the value of $\gamma(\overrightarrow{C_m}\Box \overrightarrow{C_n})$ when $m\equiv 2$ $(mod \  3)$ and we improve the lower bounds for most of the still unknown cases. We also disprove the conjectured formula appearing in \cite{Liu2} for the case $m$ $\equiv 0$ $(mod \ 3)$.

 We denote the vertices of a directed cycle $\overrightarrow{C_n}$ by $C_n=\left\{0,1,\dots,n-1 \right\}$, the integers considered modulo $n$.  Thus, when used for vertex labeling, $a+b$ and $a-b$ will denote the vertices $a+b$ and $a-b$ $(mod \ n)$. Notice that there exists an arc $xy$  from $x$ to $y$ in $\overrightarrow{C_n}$ if and only if $y\equiv x+1$ $(mod \ n)$, thus with our convention, if and only if $y= x+1$.
  For any $i$ in $\left\{0,1,\dots,n-1 \right\}$ we will denote by $\overrightarrow{C_m^i}$ the subgraph of $\overrightarrow{C_m}\Box \overrightarrow{C_n}$ induced by the vertices
 $\left\{(k,i)\: |\: k\in\left\{0,1,\dots,m-1 \right\}\right\}$. Note that $\overrightarrow{C_m^i}$ is isomorphic to $\overrightarrow{C_m}$. We will denote by $C_m^i$ the set of vertices of $\overrightarrow{C_m^i}$.
 
\section{General bounds and the case $m$ $\equiv 2$ $(mod \  3)$}
We start this section by developing a general upper bound for $\gamma(\overrightarrow{C_m}\Box \overrightarrow{C_n})$. Then we will construct minimum dominating sets for $m$ $\equiv 2$ $(mod \  3)$.
These optimal sets will be obtained from  integer solutions of a system of equations. 
  \begin{prop}\label{prop:bornecons} Let $W$ be a dominating set of $\overrightarrow{C_m}\Box \overrightarrow{C_n}$. Then for all $i$ in $\left\{0,1,\dots,n-1 \right\}$ considered modulo $n$ we have
  $\left|W\cap C_m^{i-1}\right|\:+\: 2 \left|W\cap C_m^i\right| \geq m$.
  \end{prop}  
  \begin{proof}
   The $m$ vertices of $C_m^i$ can only be dominated by vertices of $W\cap C_m^{i}$ and $W\cap C_m^{i-1}$. Each of the vertices of $W\cap C_m^{i}$ dominates two vertices in $C_m^i$. Similarly each of the vertices of $W\cap C_m^{i-1}$ dominates one vertex in $C_m^i$. The result follows.
 
  \end{proof}

\begin{theor}\label{Th:borne} 	Let $m,n \geq 2$ and $k_1=\left\lfloor \frac{m}{3}\right\rfloor $ then
\begin{enumerate}[(i)]

\item if $m$ $\equiv 0$ $(mod \  3)$ then $\gamma(\overrightarrow{C_m}\Box \overrightarrow{C_n})\geq nk_1$, or
\item if $m$ $\equiv 1$ $(mod \  3)$ then $\gamma(\overrightarrow{C_m}\Box \overrightarrow{C_n})\geq nk_1+\frac{n}{2}$, or
\item if $m$ $\equiv 2$ $(mod \  3)$ then $\gamma(\overrightarrow{C_m}\Box \overrightarrow{C_n})\geq nk_1+n$.

\end{enumerate}
  \end{theor}
  
\begin{proof}

Let $W$ be a dominating set of $\overrightarrow{C_m}\Box \overrightarrow{C_n}$ and for any $i$ in $\left\{0,1,\dots,n-1 \right\}$ let $a_i=|W\cap C_m^i|.$
Notice first, as noticed by Liu et al.(\cite{Liu2}), that each of the vertices of $W$ dominates three vertices of $\overrightarrow{C_m}\Box \overrightarrow{C_n}$ and thus $\left|W\right|\geq \frac{mn}{3}$. This general bound give the announced result  for $m=3k_1$, $\gamma(\overrightarrow{C_m}\Box \overrightarrow{C_n})\geq nk_1+\frac{n}{3}$ for $m=3k_1+1$ and  $\gamma(\overrightarrow{C_m}\Box \overrightarrow{C_n})\geq nk_1+2\frac{n}{3}$ for $m=3k_1+2$. We will improve these two last results to verify parts $(ii)$ and $(iii)$ of the theorem.

%

Assume first $m=3k_1+1$. Let $J$ be the set of $j\in\left\{0,1,\dots,n-1 \right\}$ such that  $a_j\leq k_1$. If $J=\emptyset$ then $|W|\geq n (k_1+1)\geq nk_1+\frac{n}{2}$ and we are done. Otherwise let $J'=\left\{j\;|\;j+1\:(mod \  n)\in J\right\}$.
By Proposition \ref{prop:bornecons}, for any   $i$ in $\left\{0,1,\dots,n-1 \right\}$ considered modulo $n$, we have $a_{i-1}+2a_i\geq3k_1+1$. Then if $i$ belongs to $J$,  $a_{i-1}+a_i\geq 2k_1+1$. 
A first consequence is that there are no consecutive indices, taken modulo $n$, in $J$. Indeed if $j-1$ and $j$ are in $J$ then, by definition of $J$, $a_{j-1}+a_j\leq 2k_1$ in contradiction with the previous inequality. By definition of $J'$ we have thus $J\cap J'=\emptyset$.

Now let $K =\left\{j\in\left\{0,1,\dots,n-1 \right\} \;|\;j\notin J \cup J'\right\}$.
We can write $\left\{0,1,\dots,n-1 \right\}= J \cup J'\cup K$  where $J$,$J'$ and $K$ are disjoint sets.
Notice that $\theta:j\mapsto j-1$(modulo n) induces a one to one mapping between $J$ and $J'$. 

The cardinality of $W$ is $|W|=\sum_{i\in\left\{0,1,\dots,n-1 \right\}} a_i= \sum_{i\in J} a_i+\sum_{i\in J'} a_i+\sum_{i\in K} a_i$.
We can use $\theta$ for grouping 2 by 2 the elements of  $J \cup J'$ and write $\sum_{i\in J} a_i+ \sum_{i\in J'}a_i= \sum_{i\in J} a_i+ \sum_{i\in J}a_{\theta(i)}=\sum_{i\in J} (a_i+a_{i-1})$. Using $a_{i-1}+a_i\geq 2k_1+1$, because $i\in J$, we obtain $\sum_{i\in J} a_i+\sum_{i\in J'}a_i\geq \left|J\right|(2k_1+1)$.

If $i\in K$ then $i\notin J$ and  $a_i\geq k_1+1$. Since $|K|=n-2|J|$ we have $\sum_{i\in K} a_i\geq (n-2|J|)(k_1+1)$.
Then $|W|=\sum_{i\in\left\{0,1,\dots,n-1 \right\}} a_i\geq|J|(2k_1+1)+(n-2|J|)(k_1+1)= nk_1+n-|J|$.

Since $|J|=|J'|$ and $J\cap J'=\emptyset$ , $n-|J|\geq \frac{n}{2}$ and the conclusion for (ii) follows.\\

The case $m=3k_1+2$ is similar. Let $J$ be the set of $j\in\left\{0,1,\dots,n-1 \right\}$ such that  $a_j\leq k_1$. If $J=\emptyset$ then we are done. Otherwise let $J'=\left\{j\;|\;j+1\:(mod \  n)\in J\right\}$. If $i\in J$  we have $a_{i-1}+2a_i\geq3k_1+2$ thus $a_{i-1}+a_i\geq 2k_1+2$. Then $J\cap J'=\emptyset$ and $\sum_{i\in J\cup J'} a_i\geq\left|J\right|(2k_1+2)$.  Therefore $\sum_{i\in\left\{0,1,\dots,n-1 \right\}} a_i\geq|J|(2k_1+2)+(n-2|J|)(k_1+1)\geq n(k_1+1)$.

\end{proof}

  Let us now study in detail the case $m$ $\equiv 2$ $(mod \  3)$. Assume $m=3k_1+2$. Let $A$ be the set of $k_1+1$ vertices of $\overrightarrow{C_m}$ defined by $A=\{0\}\cup\{2+3p\;|\;p=0,1,\dots,k_1-1\}=\{0\}\cup\{2,5,\dots,m-6,m-3\}$.
  For any $i$ in $\left\{0,1,\dots,m-1 \right\}$ let us call $A_i=\{j\;|\;j-i\:(mod\  m)\in A\}$  the \emph{translate}, considered modulo $m$, of $A$ by $i$. We have  thus $A_i=\{i\}\cup\{i+2,i+5,\dots,i-6,i-3\}$(see Figure \ref{fig:Ai}).
  
  We will call a set $S$ of vertices of $\overrightarrow{C_m}\Box \overrightarrow{C_n}$ a $A$-set if for any $j$ in  $\{0,1,\dots,n-1\}$ we have $S \cap C_m^j=A_{i}$ for some $i$ in  $\{0,1,\dots,n-1\}$. It will be convenient to denote this index $i$, function of $j$, as $i_j$.
 If $S$ is a $A$-set then $\left|S\right|$= $n(k_1+1)$; thus if a set is both a $A$-set and a dominating set, by Theorem \ref{Th:borne}, it is minimum and we have $\gamma(\overrightarrow{C_m}\Box \overrightarrow{C_n})= n(k_1+1)$. 
 
\begin{figure}[h]
\begin{center}
\includegraphics [scale=0.3] {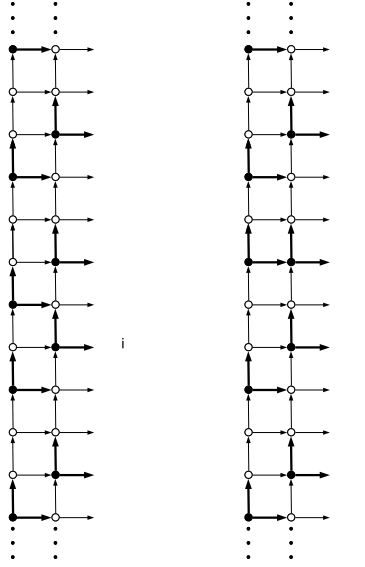}
\caption{\label{fig:Ai} $A_i$,$A_{i-1}$ and $A_{i+2}$}
\end{center}

\end{figure}

\begin{lemm}\label{le:Aset}
Let $m= 3k_1+2$. Let $S$ be a $A$-set and for any $j$ in $\{0,1,\dots,n-1\}$ define $i_j$ as the index such that $S \cap C_m^j=A_{i_j}$.
Assume that:
$$
\left\{
\begin{array}{l}
\mbox{for any }j \in \left\{1,\dots,n-1 \right\}\; i_{j}\equiv i_{j-1}+1\: (mod \  m)\mbox{ or } i_{j}\equiv i_{j-1}-2\: (mod \  m)\\
\mbox{and}\\
i_{0}\equiv i_{n-1}+1\: (mod \  m)\mbox{ or } i_{0}\equiv i_{n-1}-2\: (mod \  m).
\end{array}
\right.
$$
then $S$ is a dominating set of $\overrightarrow{C_m}\Box \overrightarrow{C_n}$.  
  \end{lemm}
  
\begin{proof}

Note first that for any $i$ in $\left\{0,1,\dots,m-1 \right\}$ the set of non dominated vertices of $C_m$ by $A_i$ is  $T=        \{i+4,i+7,\dots,i-4,i-1\}$. Note also that $A_{i+2}=\{i+2\}\cup\{i+4,i+7,\dots,i-4,i-1\}$ and $A_{i-1}=\{i-1\}\cup\{i+1,i+4,\dots,i-7,i-4\}$. Thus $T\subset A_{i+2}$ and $T\subset A_{i-1}$.

Let $j$ in $\left\{1,\dots,n-1 \right\}$. Let us prove that the vertices of $C_m^j$ are dominated. Indeed, by the previous remark and the lemma hypothesis, the vertices non dominated by  $S \cap C_m^j$ are dominated by $S \cap C_m^{j-1}$(see Figure \ref{fig:Ai}). 
For the same reasons  the vertices of $C_m^0$ are dominated by those of $S \cap C_m^0$ and $S \cap C_m^{n-1}$.

\end{proof}

We will prove next that the existence of  solutions to some system of equations over integers implies the existence of a $A$-set satisfying the hypothesis of Lemma \ref{le:Aset}.

\begin{lemm}\label{le:entiers}
Let $m= 3k_1+2$. If there exist integers $a,b\geq0$ such that
$$
\left\{
\begin{array}{l}
a+b=n-1\\
\mbox{and}\\
a-2b \equiv 2\:  (mod \  m)\: \mbox{or } a-2b \equiv m-1\:  (mod \  m)
\end{array}
\right.
$$

then $\gamma(\overrightarrow{C_m}\Box \overrightarrow{C_n})= n(k_1+1)$.  
  \end{lemm}
  
\begin{proof}
Consider a word $w= w_1\dots w_{n-1}$ on the alphabet $\{1,-2\}$ with $a$ occurrences of $1$ and $b$ of $-2$. Such a word exists, for example $w=1^a(-2)^b$. We can associate with $w$ a set $S$ of vertices of $\overrightarrow{C_m}\Box \overrightarrow{C_n}$ using the following algorithm:
\begin{itemize}
\item $S \cap C_m^0=A_0$
\item For $i=1$ to $n-1$ do\\
begin\\
Let $k$ such that $S \cap C_m^{i-1}=A_{k}$\\
If $w_i=1$ let $k'\equiv k+1$ $(mod \ m)$ else $k'\equiv k-2$ $(mod \ m)$\\
$S\cap C_m^i:=A_{k'}$\\
end
\end {itemize}
By construction $S$ is a $A$-set.
Notice that we  have $S\cap C_m^{n-1}:=A_{i_{n-1}}$ where $i_{n-1}\equiv \sum_{k=1}^{n-1}{w_k}\equiv a-2b \: (mod \ m)$.  Thus  $i_{n-1}\equiv 2\: (mod \ m)$ or  $i_{n-1}\equiv m-1\: (mod \ m)$. By Lemma \ref{le:Aset} $S$ is a dominating set.
Furthermore, because $S$ is a $A$-set, $\left|S\right|$= $n(k_1+1)$, thus  by Theorem \ref{Th:borne} it is minimum and we have $\gamma(\overrightarrow{C_m}\Box \overrightarrow{C_n})= n(k_1+1)$.
\end{proof}  
With the exception of one sub case we can find solutions $(a,b)$ of the system and thus obtain minimum dominating sets for $m\equiv2$ $mod \ 3$.

  \begin{theor}\label{th:2mod3}
  Let $m,n\geq2$ and $m\equiv2$ $mod \ 3$. Let $k_1=\left\lfloor \frac{m}{3}\right\rfloor$ and   $k_2=\left\lfloor \frac{n}{3}\right\rfloor$.
 \begin{enumerate}[(i)]
 \item if $n=3k_2$ then $\gamma(\overrightarrow{C_m}\Box \overrightarrow{C_n})= n(k_1+1)$,and
 \item if $n=3k_2+1$ and $2k_2\geq k_1$ then $\gamma(\overrightarrow{C_m}\Box \overrightarrow{C_n})= n(k_1+1)$,and
 \item if $n=3k_2+1$ and $2k_2<k_1$ then $\gamma(\overrightarrow{C_m}\Box \overrightarrow{C_n})>n(k_1+1) $,and
 \item if $n=3k_2+2 $ and $n\geq m$ then $\gamma(\overrightarrow{C_m}\Box \overrightarrow{C_n})= n(k_1+1)$,and
 \item if $n=3k_2+2 $ and $n\leq m$ then $\gamma(\overrightarrow{C_m}\Box \overrightarrow{C_n})= m(k_2+1)$.
  \end{enumerate} 
  \end{theor}

\begin{proof}
We will use Lemma \ref{le:entiers} considering the  following integer solutions of
$$
\left\{
\begin{array}{l}
a,b\geq 0 \\
a+b=n-1\\
a-2b \equiv 2\:(mod \  m) \mbox{ or } a-2b \equiv m-1 \:(mod \  m).\\
\end{array}
\right.
$$

\begin{enumerate}[(i)]
\item if $n=3k_2$ then $k_2\geq1$. Take  $a=2k_2$ and $b=k_2-1$. 
\item if $n=3k_2+1$ and $2k_2\geq k_1$ then take  $a=2k_2-k_1$ and $b=k_2+k_1$. 
\item if $n=3k_2+1$ and $2k_2 < k_1$ then $\gamma(\overrightarrow{C_m}\Box \overrightarrow{C_n})$=$\gamma(\overrightarrow{C_n}\Box \overrightarrow{C_m})\geq \frac{(2k_2+1)m}{2}$ by Theorem \ref{Th:borne}.
Furthermore $\frac{(2k_2+1)m}{2}-n(k_1+1)=\frac{k_1}{2}-k_2>0$.
\item if $n=3k_2+2$ and $k_2\geq k_1$ then take  $a=2k_2-2k_1$ and $b=k_2+2k_1+1$.
\item if $n=3k_2+2$ and $k_2\leq k_1$ then use $\gamma(\overrightarrow{C_m}\Box \overrightarrow{C_n})$=$\gamma(\overrightarrow{C_n}\Box \overrightarrow{C_m})$.

\end{enumerate}

\end{proof}

\section{The case $m$ $\equiv 0$ $(mod \  3)$}

 In \cite{Liu2} Liu et al. conjectured the following formula:
 
 \begin{conj}\label{co:1}
Let $k\geq2$. Then $\gamma(\overrightarrow{C_{3k}}\Box \overrightarrow{C_n})= k(n+1)$ for $n \not \equiv 0$ $(mod \  3)$.
 \end{conj}
 
 Our Theorem \ref{th:2mod3} confirms the conjecture when $n  \equiv 2$ $(mod \  3)$. Unfortunately the formula is not always valid when $n  \equiv 1$ $(mod \  3)$.
 
 Indeed consider $C_{3k}\Box C_4$. In \cite{Liu1} the following result is proved:
 
 \begin{theor}\label{th:liu}
Let $n\geq2$. Then
 $\gamma(\overrightarrow{C_{4}}\Box \overrightarrow{C_n})= \frac{3n}{2}$ if $n \equiv 0$ $(mod \  8)$ and 
 $\gamma(\overrightarrow{C_{4}}\Box \overrightarrow{C_n})= n+\left\lceil \frac{n+1}{2} \right\rceil$ otherwise. 
 \end{theor}
 
 We have thus $\gamma(\overrightarrow{C_{3k}}\Box \overrightarrow{C_4})=  \gamma(\overrightarrow{C_{4}}\Box \overrightarrow{C_{3k}})=3k+\left\lceil \frac{3k+1}{2}\right\rceil $ when $k \not \equiv 0$ $(mod \  8)$.
 Alternately, Conjecture \ref{co:1} proposes the value $\gamma(\overrightarrow{C_{3k}}\Box \overrightarrow{C_4})=5k$. These two numbers are different when $k\geq3$.
 \section{Conclusion}
 Consider the possible remainder of $m$, $n$ modulo 3. For some of the nine possibilities, we have found exact values for $\gamma(\overrightarrow{C_m}\Box \overrightarrow{C_n})$. The remaining cases are:
 
\begin{enumerate}[a)]
	\item $m \equiv 0$ $(mod \  3)$ and $n \equiv 1$ $(mod \  3)$ 	
	\item The symmetrical case $m \equiv 1$ $(mod \  3)$ and $n \equiv 0$ $(mod \  3)$. 
	\item $m$ and $n$ $ \equiv 1$ $(mod \  3)$.

	\item	The case  $m$ or $n$ $ \equiv 2$ $(mod \  3)$ is not completely solved by Theorem \ref{th:2mod3}. The following subcases are still open
\begin{enumerate}[i)]
	\item $m \equiv 2$ $(mod \  3)$ and $n \equiv 1$ $(mod \  3)$ with $m>2n+1$ 	
	\item the symmetrical case 	$m \equiv 1$ $(mod \  3)$ and $n\equiv 2$ $(mod \  3)$ with $n>2m+1$. 	
\end{enumerate}	
\end{enumerate}		
		For these values of $m,n$ there does not always exist a dominating set reaching the bound stated if Theorem \ref{Th:borne} and thus the determination of $\gamma(\overrightarrow{C_m}\Box \overrightarrow{C_n})$ seems to be a more difficult problem.

\end{document}